\documentclass[11pt]{amsart}

\usepackage[top=1.5in, bottom=1in, left=0.9in, right=0.9in]{geometry}

\usepackage{amscd,amsmath,amssymb,fancyhdr,color}
\usepackage[utf8]{inputenc}
\usepackage{amsfonts}
\usepackage{amsthm}
\usepackage{accents}
\usepackage{graphicx}
\usepackage{float}
\usepackage{subcaption}
\usepackage{verbatim}
\usepackage{indentfirst}
\usepackage{dsfont}
\usepackage{tikz}
\usepackage{tikz-cd}
\usetikzlibrary{matrix}
\usepackage[all]{xy}
\usepackage{enumerate}
\usepackage{extarrows}
\usepackage{csquotes}

\usepackage{scalerel,stackengine}
\stackMath
\newcommand\reallywidehat[1]{%
	\savestack{\tmpbox}{\stretchto{%
			\scaleto{%
				\scalerel*[\widthof{\ensuremath{#1}}]{\kern-.6pt\bigwedge\kern-.6pt}%
				{\rule[-\textheight/2]{1ex}{\textheight}}
			}{\textheight}%
		}{0.5ex}}%
	\stackon[1pt]{#1}{\tmpbox}%
}

\usepackage{faktor}

\usepackage{BOONDOX-uprscr}

\usepackage[backref=page]{hyperref}
\renewcommand*{\backref}[1]{}
\renewcommand*{\backrefalt}[4]{%
	\ifcase #1 (Not cited.)%
	\or        (Cited on page~#2.)%
	\else      (Cited on pages~#2.)%
	\fi}

\hypersetup{
	colorlinks   = true,
	citecolor    = magenta
}

\newcommand{\K}{K\"ahler}

\numberwithin{equation}{section}

\def\eqref#1{(\ref{#1})}

\newcommand{\C}{{\mathbb C}}
\newcommand{\R}{{\mathbb R}}

\newcommand{\del}{\partial}
\newcommand{\delb}{\overline{\partial}}

\def\1{\sqrt{-1}\:}

\newcommand{\cntrct}                
{\hspace{2pt}\raisebox{1pt}{\text{$\lrcorner$}}\hspace{2pt}}

\renewcommand{\dim}{\operatorname{dim}}

\newcommand{\Img}{\operatorname{Im}}

\newcommand{\ie}{{\em i.e. }}
\newcommand{\eg}{{\em e.g. }}

\renewcommand{\to}{\longrightarrow}

\newcounter{Mycounter}[section]
\newcounter{lemma}[section]
\newcounter{claim}[section]
\newcounter{sublemma}[section]
\newcounter{corollary}[section]
\newcounter{theorem}[section]
\newcounter{conjecture}[section]
\newcounter{proposition}[section]
\newcounter{definition}[section]
\newcounter{example}[section]
\newcounter{remark}[section]
\newcounter{problem}[section]
\newcounter{question}[section]
\makeatletter

\@addtoreset{equation}{section}

\@addtoreset{footnote}{section}

\makeatother

\makeatletter
\DeclareRobustCommand*{\mfaktor}[3][]
{
	{ \mathpalette{\mfaktor@impl@}{{#1}{#2}{#3}} }
}
\newcommand*{\mfaktor@impl@}[2]{\mfaktor@impl#1#2}
\newcommand*{\mfaktor@impl}[4]{
	\settoheight{\faktor@zaehlerhoehe}{\ensuremath{#1#2{#3}}}%
	\settoheight{\faktor@nennerhoehe}{\ensuremath{#1#2{#4}}}%
	\raisebox{-0.5\faktor@zaehlerhoehe}{\ensuremath{#1#2{#3}}}%
	\mkern-4mu\diagdown\mkern-5mu%
	\raisebox{0.5\faktor@nennerhoehe}{\ensuremath{#1#2{#4}}}%
}
\makeatother

\usetikzlibrary{arrows,chains,matrix,positioning,scopes}

\makeatletter
\tikzset{join/.code=\tikzset{after node path={%
			\ifx\tikzchainprevious\pgfutil@empty\else(\tikzchainprevious)%
			edge[every join]#1(\tikzchaincurrent)\fi}}}
\makeatother

\tikzset{>=stealth',every on chain/.append style={join},
	every join/.style={->}}

\makeatletter
\newtheorem*{rep@theorem}{\rep@title}
\newcommand{\newreptheorem}[2]{%
	\newenvironment{rep#1}[1]{%
		\def\rep@title{\ref{##1}}%
		\begin{rep@theorem}}%
		{\end{rep@theorem}}}
\makeatother

\newreptheorem{theorem}{Theorem}

\begin{document}
	
	\newpage
	
	\title[On products of locally conformally K\" ahler manifolds]{On products of locally conformally K\" ahler manifolds}

  	\author{Miron Stanciu}
	\address{Miron Stanciu \newline
		\textsc{\indent University of Bucharest, Faculty of Mathematics and Computer Science\newline 
			\indent 14 Academiei Str., Bucharest, Romania \newline
			\indent \indent and \newline
			\indent Institute of Mathematics ``Simion Stoilow'' of the Romanian Academy\newline 
			\indent 21 Calea Grivitei Street, 010702, Bucharest, Romania}}
	\email{miron.stanciu@fmi.unibuc.ro; miron.stanciu@imar.ro}
	
	\thanks{%
		\textbf{Keywords:} Locally conformally \K, lcK product, lcK with potential. \\
		\hspace*{\parindent}\textbf{2020 Mathematics Subject Classification:} 53C55, 32J27}
	
	\date{\today}

	\begin{abstract}
		We prove that no product of positive-dimensional compact complex manifolds admits a strict locally conformally \K \ (lcK) metric. This solves a long-standing conjecture in lcK geometry, previously established only under certain conditions.
	\end{abstract}
	
	\maketitle
	
	\hypersetup{linkcolor=blue}
	\tableofcontents

	\section{Introduction}
	\label{sec:introduction}
	
	A locally conformally \K \ (lcK) manifold is a Hermitian manifold $(M, J, g)$ such that the fundamental form $\omega = g(J \cdot, \cdot)$ satisfies the differential equation $d\omega = \theta \wedge \omega$ for some closed $1$-form $\theta$, called the Lee form; $\omega$ is then usually referred to as an lcK metric. If $\theta$ is not exact, $\omega$ is called strict lcK. Their systematic study was started by Vaisman \cite{vai76}. In \cite{vai80} he proved that K\" ahler and strict lcK metrics cannot coexist on a compact manifold with respect to the same complex structure. This is one motivation for lcK geometry being studied primarily in the compact case, as there we have a clear separation between \K \ and lcK geometry. A striking fact is that, as opposed to the \K \ case, few general topological obstructions to the existence of strict lcK metrics are known in arbitrary dimension, beyond the necessary condition $b_1 \neq 0$. 
	
	In terms of examples, many non-Kähler compact surfaces are lcK (see \cite[Section 2.4]{ovv24}). In higher dimensions, Hopf manifolds (\cite{ov23}), certain higher-dimensional Kato manifolds (\cite{iop21}) and some Oeljeklaus-Toma (OT) manifolds (\cite{dv23}) are lcK. For a recent comprehensive study on the state of the art in the field, see \cite{ovbook}.
	
	However, while great progress has been made since Vaisman's papers in understanding lcK geometry, it is often the case that questions that have an immediate or straightforward answer in the \K \ setting prove very difficult to solve in full generality on lcK manifolds. For instance, while the stability of a \K \ structure under blow-up is a classical result (\cite{bla56}), it took rather longer and some intermediate results to obtain the precise conditions under which the same is true for an lcK structure (\cite{ovv13}, see also \cite{ps23a}).
	
	An even more elementary question that has proved very difficult is whether lcK structures are ever compatible with taking products. Two facts are immediately clear: first, that the product of two \K \ manifolds is \K, with the sum of the pullbacks of the \K \ forms on the two factors. Second, that the same naive construction does not work when even one of the terms is strict lcK instead of \K. One may therefore ask whether any product of compact lcK manifolds can be endowed with \textit{any} strict lcK metric (compare this with the locally conformally symplectic case, where the answer is affirmative simply because any compact almost complex manifold with nonzero integral $1$-cohomology admits strict lcs structures by \cite{em23}). Since the existence of such a metric would automatically imply that both factors are lcK anyway via restriction, the conjecture answering the question in the negative is usually formulated as follows:
	
	\begin{conjecture}
		If $X$ and $Y$ are compact complex manifolds of positive dimension, then $X \times Y$ admits no strict lcK metric.
	\end{conjecture}
	
	\smallskip
	
	The present paper proves the above conjecture in \ref{thmprodus}.
	
	\smallskip
	
	There have been several partial results related to the conjecture, one of which will be crucial for our proof. In \cite[Corollary 3.3]{ts99}, Tsukada proved that a product between compact Vaisman manifolds (a particular type of lcK manifold) cannot be strict lcK. The conjecture was also known for a product between a compact strict lcK manifold and a compact \K \ manifold of dimension at least $2$ (\cite[Corollary 2]{opv14}) and for the product of a compact complex curve with a compact strict lcK manifold admitting no lcK metric with potential (\cite[Proposition 7]{is19}).
	
	The strongest result to date is given in \cite{ovv24}. There, the authors show that all (at the time) \textit{known} examples of lcK manifolds fall into at least one of three classes: with potential, Inoue-type or containing rational curves. They then prove separately that no product between any compact complex manifold of positive dimension and an lcK manifold belonging to one of the three classes can be strict lcK. This has the effect of proving the conjecture for all up-to-then known examples of lcK manifolds, but it should be noted that there is no result suggesting that any lcK manifold must fall into one of the categories above. In fact, even for surfaces, the argument for the classification being exhaustive is conditional on the Global Spherical Shell (GSS) conjecture being true. 
	
	The idea of \ref{thmprodus} is to reduce the general case to the one where one of the factors is lcK with potential and then apply the product obstruction given in \cite{ovv24}. We do this by taking a hypothetical strict lcK metric on the product and looking at the family of metrics induced on one factor by restriction and their Bott-Chern classes. From this family we are either able to construct an lcK metric on one of the factors with vanishing Bott-Chern class, which leads to an lcK with potential structure, or we use a convex separation argument and an obstruction result \ref{lem} to also arrive at a contradiction.
	
	\medskip
	
	In Section \ref{sec:prelim}, we recall all necessary definitions and go into some more detail on the objects discussed above. Section \ref{sec:main} contains the proof of the main result.
	
	\section{Preliminaries}
	\label{sec:prelim}
	\textbf{Conventions.} Throughout this paper, all manifolds are assumed connected, without boundary. All covering spaces are implicitly Galois, so the deck group acts transitively on the fibers. 
	
	We shall use the conventions from \cite[Chapter 2.A]{bes87} for the complex structure $J$ acting on complex forms on a manifold $(M, J)$. Namely:
	\begin{itemize}
		\item $J\alpha=\mathrm{i}^{q-p} \alpha$, for any $\alpha \in \Omega^{p, q}M$, or equivalently, $J\eta(X_1, \ldots, X_k) = (-1)^k\eta(JX_1, \ldots, JX_k)$;
		\item the fundamental form of a Hermitian metric is given by $\omega (X, Y) := g(JX, Y)$;
		\item the operator $d^c$ is defined as $d^c := -J^{-1}dJ$, where $J^{-1} \alpha = (-1)^{\deg \alpha} J \alpha$. Hence $dd^c = 2i \del \delb$.
	\end{itemize}
	
	\bigskip
	
	\textbf{Locally conformally \K \ manifolds.} We will now recall some basic theory about locally conformally \K \ (lcK) manifolds. There are several equivalent definitions for what an lcK manifold is, and we will use them interchangeably in this paper.
	
	\begin{definition}
		Let $M$ be a complex manifold of complex dimension $n$. As the name implies, a Hermitian form $\omega \in \Omega^{1, 1}(M)$ is called lcK if it admits an open cover $(U_i)_{i \in I}$ and functions $f_i \in \mathcal{C}^\infty(U_i)$ such that $e^{-f_i} \omega$ is \K \ on $U_i$.
	\end{definition}
	
	Although under the above conditions the functions $f_i$ do not necessarily glue together, if $n \ge 2$, it is immediate that $df_i = df_j$ on $U_i \cap U_j$, so they determine a closed (real) $1$-form $\theta$ on $M$. The \K \ condition above is then seen to be equivalent to the formulation below, which is probably the most widespread:
	
	\begin{definition}
		\label{deflck}
		A Hermitian $(1, 1)$-form $\omega$ on $M$ is called an lcK form if there exists a $\theta \in \Omega^1(M)$ with $d\theta = 0$ such that $d\omega = \theta \wedge \omega$.
		
		It is easy to see that, in complex dimension $n \ge 2$, the form $\theta$, if it exists, is uniquely determined by $\omega$. It is called \textit{the Lee form} of $\omega$.
	\end{definition}
		
	\begin{remark}
		The lcK condition is conformally invariant. Indeed, if $\omega$ is lcK with Lee form $\theta$ and $f \in \mathcal{C}^\infty(M)$, then $e^f \omega$ is again lcK with Lee form $\theta + df$.
		
		This shows that, for dimension $n \ge 2$, an lcK metric is \textit{globally conformally \K \  (gcK) } if and only if $\theta$ is exact. If it is not, we call $\omega$ \textit{strict lcK}.
	\end{remark}
	
	\smallskip 
	
	For any closed $1$-form on any complex manifold, it is useful to consider the differential operator that is implicit in \ref{deflck}. We define the operator
	\[
	d_\theta: \Omega^k(M) \to \Omega^{k+1}(M), \ d_\theta = d - \theta \wedge \cdot \ .
	\]
	The lcK condition can then be rewritten as $d_\theta \omega = 0$. Moreover, since $\theta$ is closed, $d_\theta^2 = 0$, so one can consider its cohomology (called the Morse-Novikov cohomology). The operator $d_\theta$ can also be seen as a flat connection in a topologically trivial line bundle over $M$, called the \textit{weight bundle} (see \cite[3.4.2]{ovbook}).
	
	The decomposition $d = \del + \delb$ induces a similar splitting for the operator $d_\theta$:
	\[
	d_\theta = \del_\theta + \delb_\theta, \text{ where } \del_\theta = \del - \theta^{1, 0} \wedge \cdot \text{ and } \delb_\theta = \delb - \theta^{0, 1} \wedge \cdot \ ;
	\]
	clearly $\del_\theta (\Omega^{p, q}(M)) \subset \Omega^{p+1, q}(M)$ and $\delb_\theta (\Omega^{p, q}(M)) \subset \Omega^{p, q+1}(M)$. The operators $d_\theta, \del_\theta, \delb_\theta$ are usually called \textit{twisted} differential operators.
	
	\medskip 
	
	It is not very difficult to prove the following alternative characterization in terms of covering spaces (see \eg \cite[Theorem 3.31]{ovbook}):
	
	\begin{proposition}
		\label{def:lckcover}
		Let $M$ be a complex manifold. Then $M$ is lcK if and only if there exists a \K \ covering $(\tilde{M}, \tilde{\omega})$  such that the deck group acts on $\tilde{\omega}$ by homotheties \ie $\gamma^* \tilde{\omega} = c_\gamma \tilde{\omega}$, $c_\gamma > 0$.
	\end{proposition}
	
	\smallskip
	
	Note that the covering in the preceding proposition can always be taken to be the universal cover, but for many lcK manifolds there are intermediate coverings that also satisfy the condition; in fact, the smallest is the minimal covering on which the pullback of $\theta$ is exact. We call this \textit{the minimal \K \ cover}.
	
	\medskip 
	
	We now give the definition of a particular type of lcK manifold, introduced by Ornea and Verbitsky:
	
	\begin{definition}
		\label{def:lckpot}
		An lcK manifold $(M, \omega)$ is called \textit{lcK with potential} if it has a \K \ covering $(\tilde{M}, \tilde{\omega})$ such that $\tilde{\omega} = i \del \delb \varphi$ for some smooth $\varphi: \tilde{M} \to \R_{>0}$ on which the deck group acts by homotheties \ie $\gamma^* \varphi = c_\gamma \varphi$. In this case, the constants $c_\gamma$ are clearly the same as in \ref{def:lckcover}.
	\end{definition}
	
	\smallskip
	
	One reason to study lcK manifolds with potential is that while lcK structures are not stable under small deformations, it turns out that the smaller class of lcK manifolds with potential does have this property. They are also a generalization of Vaisman manifolds, a well-studied subclass of lcK manifolds. For a comprehensive study of the known results on lcK manifolds with potential, see \cite[Chapter 12]{ovbook}.
	
	
	\bigskip
	
	\textbf{The twisted Bott-Chern groups.} We now give the definition of some crucial objects in the main proof, namely the (twisted) Bott-Chern groups, and fix some notations that we will use later.
	
	Let $M$ be a complex manifold and $\theta \in \Omega^1(M)$ be a closed real $1$-form. The $\theta$-twisted Bott-Chern cohomology groups are defined by replacing the usual differential operators with their twisted counterparts in the standard definition:
	\[
	H^{p, q}_{BC, \theta} = \frac{\{\eta \in \Omega^{p, q}(M) \ | \ \del_\theta \eta = \delb_\theta \eta= 0\} }{\del_\theta \delb_\theta \left( \Omega^{p-1, q-1}(M) \right)}.
	\]
	If $M$ is compact, essentially the same elliptic theory argument as for the usual Bott-Chern cohomology shows that this is a finite-dimensional complex vector space (\cite[Theorem 18.35]{ovbook}).
	
	It will be more convenient to work with the real (\ie fixed by conjugation) part of the above space, and we will only need to consider it in bidegree $(1, 1)$ for what follows:
	\begin{equation}
		\label{eq:notatiespatiu}
		V_\theta := \left(H^{1, 1}_{BC, \theta}\right)_\R = \frac{\{\eta \in \Omega^{1, 1}(M) \ | \ \eta \text{ real, } \del_\theta \eta = \delb_\theta \eta = 0\} }{i \del_\theta \delb_\theta \left( \mathcal{C}^\infty(M, \R)\right)}.
	\end{equation}
	Clearly $V^\C_\theta \simeq H^{1, 1}_{BC, \theta}$ so, by the above, if $M$ is compact, $V_\theta$ is a finite-dimensional real vector space. Moreover, with the notations
	\begin{equation}
		\label{eq:notatiespatiu2}
		Z_\theta = \{\eta \in \Omega^{1, 1}(M) \ | \ \eta \text{ real, } \del_\theta \eta = \delb_\theta \eta = 0\}, \ B_\theta = i \del_\theta \delb_\theta \left( \mathcal{C}^\infty(M, \R)\right), \text{ hence  } V_\theta = Z_\theta / B_\theta,
	\end{equation}
	we have that $Z_\theta$ is a Fr\' echet space (being closed in a Fr\' echet space of smooth forms, see \eg \cite[Part II, Corollary 1.3.9]{ham82}) and, if $M$ is compact, $B_\theta$ is closed in $Z_\theta$ (\cite[Theorem 18.15]{ovbook}). It then follows from the open mapping theorem that the Fr\' echet topology on $V_\theta$ is the same as the usual finite-dimensional topology; in particular, $q: Z_\theta \to V_\theta$ is continuous.
	
	\medskip
	
	We will work with the above space of Bott-Chern cohomology classes in the proof of \ref{thmprodus}, where the aim will be to prove a factor of the product admits an lcK form with vanishing Bott-Chern class. This is justified by a theorem we reproduce below:
	
	\begin{theorem}(\cite[Theorem 18.37]{ovbook})
		\label{thm:bc0pot}
		Let $(M, \omega, \theta)$ be a compact lcK manifold. Suppose $[\omega]_{BC} = 0$. Then $M$ admits a (perhaps different) lcK structure with potential.
	\end{theorem}

	\begin{remark}
		\label{rmk:potestrictlck}
		Note that if $(M, \omega, \theta)$ is a compact lcK manifold with $[\omega]_{BC} = 0$, it is automatically strict lcK. Indeed, if $\omega = i\del_\theta \delb_\theta \psi$ for some $\psi: M \to \R$ and $\theta = df$, then we immediately have that $i\del \delb (e^{-f} \psi) = e^{-f} \omega =: \omega'$, which is then a $dd^c$-exact \K \ form. But this is impossible due to Stokes' theorem:
		\[
		0 < \int_M (\omega')^n = \int_M d(\eta \wedge (d\eta)^{n-1}) = 0,
		\]
		where $\eta = \frac{1}{2}d^c (e^{-f} \psi)$ and $n$ is the complex dimension of $M$.
	\end{remark}
	
	\bigskip
	
	\textbf{A partial non-existence result.} As mentioned in the introduction, the most general result proved up to now, due to Ornea-Verbitsky-Vuletescu and comprised of several theorems treating different cases, shows that no product between any compact complex manifold and any \textit{known} example of compact lcK manifold can be strict lcK. We end this section by reproducing the theorem that we need in our main proof:
	
	\begin{theorem}(\cite[Theorem 3.3]{ovv24})
		\label{thm:ovv}
		Let $X$ be a compact lcK manifold with potential and $Y$ any compact complex manifold, $\dim Y > 0$. Then the product $X \times Y$ does not admit a strict lcK structure.
	\end{theorem}
		
	\section{The main result}
	\label{sec:main}
	
	We first prove an obstruction that will prove crucial for our main theorem:
	
	\begin{lemma}
		\label{lem}
		Let $(M, \omega, \theta)$ be a compact, strict lcK manifold of complex dimension $n$. Then $M$ admits no real positive twisted pluriharmonic function \ie $f: M \to \R$, $f > 0$, such that $\del_\theta \delb_\theta f = 0$.
		
		\begin{proof}
			Choose a covering $\tilde{M} \xrightarrow{\pi} M$ such that $\pi^* \theta = d\rho$, so $e^{-\rho} \pi^* \omega$ is \K \ on $\tilde{M}$.
			
			Assume such an $f > 0$ exists and take $F: \tilde{M} \to \R$, $F = e^{-\rho} \pi^* f$. It follows that $\pi^* (\del_\theta \delb_\theta f) = e^\rho \del \delb F$, hence $\del \delb F = 0$; equivalently, $dd^c F = 0$.
			
			Now consider the lcK form $\omega' = f^{-1} \omega$ on $M$, with Lee form $\theta' = \theta - d \log f$. We have that $\pi^* \theta' = d(\rho - \log (e^\rho F)) = -d \log F = - F^{-1} dF$, so $d^c F = JdF = -F J (\pi^* \theta')$. It follows that 
			\[
			0 = d d^c F = -dF \wedge J(\pi^* \theta') - F dJ(\pi^* \theta') = F \pi^* \theta' \wedge J(\pi^* \theta') - F dJ(\pi^* \theta'),
			\]
			hence $\theta' \wedge J \theta' = dJ \theta'$ on $M$. Using this, we get that
			\begin{equation}
				\label{eqptStokes}
				\begin{split}
					d\left( J \theta' \wedge (\omega')^{n-1} \right) &= \theta' \wedge J \theta' \wedge (\omega')^{n-1} - (n-1) J \theta' \wedge \theta' \wedge (\omega')^{n-1} = n \theta' \wedge J\theta' \wedge (\omega')^{n-1} \\ &= \| \theta' \|_{\omega'}^2 (\omega')^n,
				\end{split}
			\end{equation}
			where the last equality is a standard identity true for any real $1$-form (see \eg \cite[Proposition 1.2.31]{huy05}).
			
			Integrating \eqref{eqptStokes} on $M$ and using Stokes' theorem, we obtain that $\theta' = 0$ \ie $\theta = d \log f$, a contradiction with our assumption that $\omega$ was strict lcK. 
		\end{proof}
	\end{lemma}
	
	\bigskip
	
	We can now prove the main theorem:
	
	\begin{theorem}
		\label{thmprodus}
		Let $X, Y$ be compact complex manifolds of positive dimension. Then $X \times Y$ admits no strict lcK metric.
		
		\begin{proof}
			The case where one of the factors is $1$-dimensional is known (\cite[Corollary 3.4]{ovv24}), so we may assume $\dim_\C X$, $\dim_\C Y \ge 2$.
			
			Assume $X \times Y$ carries a strict lcK metric $\Omega$ with Lee form $\Theta$. By the (real) K\" unneth formula, $\Theta = p^*_X \alpha + p^*_Y \beta + dh$, where $\alpha$ and $\beta$ are closed $1$-forms on $X$ and $Y$ respectively. By multiplying $\Omega$ with $e^{-h}$, we may assume that 
			\begin{equation}
				\label{eq:Theta}
				\Theta = p^*_X \alpha + p^*_Y \beta.
			\end{equation}
			Since $\Theta$ is not exact, at least one of $\alpha$ and $\beta$ is not exact; we assume that $[\beta]_{\textrm{dR}} \neq 0$ on $Y$.
			
			The $(1, 1)$ lcK form $\Omega$ admits a decomposition with respect to the product structure; denoting by $\Omega^{p, q; r, s}(X \times Y)$ the space of forms with bidegree $(p, q)$ in the $X$-coordinates and $(r, s)$ in the $Y$-coordinates, we can write $\Omega = A + B + C + D$ uniquely, where $A \in \Omega^{1, 1; 0, 0}(X \times Y), B \in \Omega^{1, 0; 0, 1}(X \times Y)$, $C \in \Omega^{0, 1; 1, 0}(X \times Y)$ and $D \in \Omega^{0, 0; 1, 1}(X \times Y)$.
			
			The decomposition of the Lee form \eqref{eq:Theta} also determines a splitting of the associated twisted operators: $\del_\Theta = \del^X_\alpha + \del^Y_\beta$ and $\delb_\Theta = \delb^X_\alpha + \delb^Y_\beta$, where for example $\del_\alpha^X = (\del^X - \alpha^{1, 0} \wedge \cdot )$ and $\del^X$ is taken to mean the derivative with respect to the $X$-coordinates. Note that any two such operators with respect to different factors anti-commute. The usual lcK conditions $\del_\Theta \Omega = 0$ and $\delb_\Theta \Omega = 0$ can then be written equivalently as:
			\begin{equation*}
				\begin{cases}
					\del^X_\alpha A &= 0 \\
					\del^X_\alpha B &= 0 \\
					\del^Y_\beta A + \del^X_\alpha C &= 0 \\
					\del^Y_\beta B + \del^X_\alpha D &= 0 \\
					\del^Y_\beta C &= 0 \\
					\del^Y_\beta D &= 0
				\end{cases}
				\ \ \text{ and } \ \
				\begin{cases}
					\delb^X_\alpha A &= 0 \\
					\delb^X_\alpha C &= 0 \\
					\delb^Y_\beta A + \delb^X_\alpha B &= 0 \\
					\delb^Y_\beta C + \delb^X_\alpha D &= 0 \\
					\delb^Y_\beta B &= 0 \\
					\delb^Y_\beta D &= 0
				\end{cases}.
			\end{equation*} 
			From these identities, we observe two facts about the real form $A \in \Omega^{1, 1; 0, 0}(X \times Y)$. Firstly, note that $d^X_\alpha A = 0$, so one may consider its Bott-Chern class on the manifold $X \times \{y\} \simeq X$ for any level $y \in Y$. Secondly, 
			\[
			0 = \del^Y_\beta \left( \delb^Y_\beta A + \delb^X_\alpha B \right) = \del^Y_\beta \delb^Y_\beta A +  \del^Y_\beta \delb^X_\alpha B =  \del^Y_\beta \delb^Y_\beta A - \delb^X_\alpha \del^Y_\beta B = \del^Y_\beta \delb^Y_\beta A + \delb^X_\alpha \del_\alpha^X D,
			\]
			so
			\begin{equation}
				\label{eq:AsiD}
				\del^Y_\beta \delb^Y_\beta A = \del^X_\alpha \delb_\alpha^X D.
			\end{equation}
			
			We use the notations from \eqref{eq:notatiespatiu} and \eqref{eq:notatiespatiu2}:
			\[
			Z_\alpha := \{\eta \in \Omega^{1, 1}(X) \ | \ \eta \text{ real, } \del_\alpha \eta = \delb_\alpha \eta = 0\}, \ B_\alpha := i \del_\alpha \delb_\alpha \left( \mathcal{C}^\infty(X, \R)\right), \ V_\alpha := Z_\alpha / B_\alpha;
			\]
			since $X$ is compact, the projection map $q: Z_\alpha \to V_\alpha$ is then continuous between a Fr\' echet and a finite-dimensional space. Take $\Phi: Y \to Z_\alpha, \Phi(y) = A_{| X \times \{y\}}$, which is clearly smooth, and consider the map $\varphi: Y \to V_\alpha$, $\varphi = q \circ \Phi$. 
			
			We now see $\Phi$ and $D$ as forms on $Y$ with values in Fr\' echet spaces, namely $\Phi \in  \Omega^{0,0}(Y, Z_\alpha^\C)$ and $D \in \Omega^{1, 1}(Y, \mathcal{C}^\infty(X, \C))$. The equation \eqref{eq:AsiD} then becomes $\del^Y_\beta \delb^Y_\beta \Phi = \del^X_\alpha \delb_\alpha^X D$ in the space $\Omega^{1, 1}(Y, Z_\alpha^\C)$.
			
			Extending $q$ $\C$-linearly and acting component-wise on $Z_\alpha$-valued forms, we can apply it to this last relation to get $q(\del^Y_\beta \delb^Y_\beta \Phi) = q (\del^X_\alpha \delb_\alpha^X D)$ in $\Omega^{1, 1}(Y, V_\alpha^\C)$. As $q$ commutes with the differential operators in the $Y$ direction and since $q(\del^X_\alpha \delb_\alpha^X D) = 0$ due to the definition of $V_\alpha$, we get
			\begin{equation}
				\label{eq:pluriarm}
				\del_\beta \delb_\beta \varphi = 0,
			\end{equation}
			\ie each component of $\varphi$ as a real vector-valued function is $\del_\beta \delb_\beta$-closed. 
			
			\medskip
			
			We now consider the image of $\varphi$: since $Y$ is compact, $\Img \varphi \subset V_\alpha$ is compact and so is its convex hull, which we denote by $K$. Note that all Bott-Chern classes in $\Img \varphi$ have a positive representative by definition. We distinguish two cases:
			
			\textbf{Case 1:} If $0 \notin K$, by Hahn-Banach, we can pick a linear map $T: V_\alpha \to \R$ such that $T$ is strictly positive on $K$. Then, by \eqref{eq:pluriarm}, the function $f = T \circ \varphi$ satisfies $\del_\beta \delb_\beta f = 0$ and $f > 0$, contradicting \ref{lem}, since we assumed $[\beta] \neq 0$.
			
			\textbf{Case 2:} If $0 \in K$, there exist $y_1, ..., y_N \in Y$ and $t_1, ..., t_N \ge 0$ with $\sum\limits_{i=1}^N t_i = 1$ such that $\sum\limits_{i=1}^N t_i \varphi(y_i) = 0 \in V_\alpha$; by the definition of $V_\alpha$, this means there exist $\omega_1, ..., \omega_N \in \Omega^{1, 1}(X)$ positive $d_\alpha$-closed Hermitian metrics such that $\omega := \sum\limits_{i=1}^N t_i \omega_i$, which is also a positive $d_\alpha$-closed Hermitian metric, is $\del_\alpha \delb_\alpha$-exact. By \ref{rmk:potestrictlck}, $[\alpha] \neq 0$ and, by \ref{thm:bc0pot}, $X$ admits an lcK with potential metric. This reverts the situation to the known case where one of the factors of $X \times Y$ is with potential, and we obtain a contradiction by \ref{thm:ovv}, thus concluding the proof.
		\end{proof}
		
	\end{theorem}
  
	\bigskip
	
	We end by mentioning a natural extension of \ref{thmprodus} that would now be interesting to consider. In \cite{ps23b} and \cite{ps26}, the authors proved a generalization of the Vaisman theorem for lcK spaces with singularities, showing that the separation between \K \ and strict lcK carries over to compact complex spaces, under certain conditions. Therefore, it is natural to now ask:
	
	\smallskip
	
	\begin{question}
		If $X, Y$ are positive dimensional compact complex spaces, can $X \times Y$ admit a strict lcK structure?
	\end{question}
	
	\smallskip
	
	Obviously, in light of this paper, at least one of the factors would have to be singular. Notably, since the Vaisman theorem proved in \cite{ps23b}, \cite{ps26} is valid under certain conditions (for example, local ireducibility), and counterexamples exist for the general case (\cite[Examples 4.5]{ps23b}), a negative answer to the above question is not likely to be obtained by simply reducing it to the smooth case and might require new methods.
	
	\bigskip
	
	\textbf{Acknowledgements.} The author is grateful to Liviu Ornea and Victor Vuletescu for a careful reading of this manuscript before publication.
	
	\bigskip
	
	\textbf{Use of AI tools.} The author acknowledges the use of general-purpose large language models as an aid for searching the literature, discussing proof strategies and for proof checking. No AI assistance was used in writing this manuscript and the author takes responsibility for its correctness.

\end{document}